\documentclass{amsart}
\usepackage{amsfonts, amsthm, amssymb, amsmath, stmaryrd}
\usepackage{mathrsfs,array}
\usepackage{eucal,times,color,enumerate,accents}
\usepackage[all]{xy}
\usepackage{url}
\usepackage{enumitem}
\usepackage{amssymb}

\usepackage{enumitem}% http://ctan.org/pkg/enumitem   allows nested enumerations
\usepackage[normalem]{ulem}

\usepackage{placeins} %for objects not to float: \FloatBarrier

\usepackage{comment} %to comment out whole sections

\usepackage{graphicx}
\usepackage{tikz-cd}
\usetikzlibrary{calc}
\usetikzlibrary{matrix,arrows,decorations.pathmorphing}

\input xy
\xyoption{all}

\newcommand{\calO}{\mathcal{O}}

\newcommand{\calC}{\mathcal{C}}

\newcommand{\bbZ}{\mathbb{Z}}

\newcommand{\bbC}{\mathbb{C}}
\newcommand{\bbQ}{\mathbb{Q}}

\newcommand{\Hom}{\mathrm{Hom}}
\newcommand{\Ext}{\mathrm{Ext}}

\newcommand{\End}{\mathrm{End}}

\newcommand{\GL}{\mathrm{GL}}

\newcommand{\bk}{\bar{k}}

\newcommand{\Id}{\mathrm{Id}}
\newcommand{\bbN}{\mathbb{N}}

\newcommand{\rarr}{\rightarrow}

\newcommand{\bbF}{\mathbb{F}}
\newcommand{\bbbF}{\bar{\mathbb{F}}}
\newcommand{\frakQ}{\mathfrak{Q}}
\newcommand{\frakU}{\mathfrak{U}}
\newcommand{\frakp}{\mathfrak{p}}

\newcommand{\Def}{\mathrm{Def}}

\newcommand{\brho}{\bar{\rho}}
\newcommand{\bV}{\bar{V}}

\newcommand{\Ind}{\mathrm{Ind}}

\newcommand{\Tr}{\mathrm{Tr}}

\newcommand{\SL}{\mathrm{SL}}

\newcommand{\bd}{\mathbf{d}}
\newcommand{\bu}{\mathbf{u}}

\newcommand{\hcalC}{\hat{\mathcal{C}}}
\newcommand{\ksi}{\xi}

\begin{document}
\bibliographystyle{alpha}
\newtheorem{thm}{Theorem}[section]
\newtheorem*{thm*}{Theorem}
\newtheorem{proposition}[thm]{Proposition}
\newtheorem{lemma}[thm]{Lemma}
\newtheorem{cor}[thm]{Corollary}
\newtheorem{claim}[thm]{Claim}
\newtheorem{conjecture}[thm]{Conjecture}
\newtheorem{prop}[thm]{Proposition}

\theoremstyle{definition}
\newtheorem{question}[thm]{Question}
\newtheorem{answer}[thm]{Answer}
\newtheorem{remark}[thm]{Remark}
\newtheorem{example}[thm]{Example}
\newtheorem{warning}[thm]{Warning}
\newtheorem{notation}[thm]{Notation}
\newtheorem{construction}[thm]{Construction}
\newtheorem{fact}[thm]{Fact}

\newtheorem{defn}[thm]{Definition}
\newtheorem{obs}[thm]{Observation}
\newtheorem{rmk}[thm]{Remark}
\newtheorem{ex}[thm]{Example}
\newtheorem{quest}[thm]{Question}

\newcommand{\adjunction}[4]{\xymatrix@1{#1{\ } \ar@<-0.3ex>[r]_{ {\scriptstyle #2}} & {\ } #3 \ar@<-0.3ex>[l]_{ {\scriptstyle #4}}}}

\title{On the inverse problem for deformations of finite group representations}
\author{Marcin Lara}

\begin{abstract}
Let $s$ be even and $q=p^s$. We show that the ring $W(\bbF_{q})[\![X]\!]/(X^2-pX)$ is a quotient of the universal deformation ring of a representation of a finite group. This amounts to giving an example of a finite group and its $\bbF_q$-representation that lifts to $W(\bbF_q)$ in two different ways and satisfies certain subtle extra conditions. We achieve this by studying representations of $\SL(2,\bbF_{p^2})$.
\end{abstract}

\maketitle
\tableofcontents
%\pagebreak

\section*{Introduction}
\subsubsection*{Inverse problem, previous work and unanswered questions} Fix a finite field $k$ and consider the category $\hcalC$ of all complete noetherian local rings with residue field $k$. The inverse problem in the theory of deformations of group representations asks which rings in the category $\hcalC$ may occur as universal deformation rings for some representation over $k$ of a profinite group. After some partial results, which have shown that probably more rings can occur than initially expected, the final answer was given by Dorobisz in \cite{dorobisz2} (and independently by Eardley and Manoharmayum in \cite{EaM}). It turns out that all the rings in $\hcalC$ can be obtained as universal deformations rings for some profinite group and its representation. The construction is quite uniform -- for a ring $R$, one considers the profinite group $\SL_n(R)$ (assume $n>3$ for simplicity) and its natural action on $k^n$ given by reducing the coefficients of $R$ modulo its maximal ideal. It turns out that this representation admits $R$ as its universal deformation ring.

This leads to another question: which rings may occur for a representation of a finite group $G$? The above discussion shows that all the finite rings in $\hcalC$ do occur, but it is easy to see that this cannot possibly be the full answer. For example, one can get $\bbZ_p$. In his work, Dorobisz rules out a large number of infinite rings and asks what is true for the remaining ones.
For example, he asks whether rings of the form $W(k)[\![X]\!]/(X^2-p^rX)$ can be obtained in this manner. Here $k$ is a finite field and $W(k)$ is the ring of Witt vectors over $k$.

As a first step, one can start with the following question: which rings of $\hcalC$ occur as quotients of universal deformation rings of representations of finite groups? We denote the class of such rings by $\frakQ$. Similarly, Dorobisz asks whether $W(k)[\![X]\!]/(X^2-p^rX)$ is in $\frakQ$.

\subsubsection*{Main results}
The goal of this article is to give a solution to the problem whether the ring $W(k)[\![X]\!]/(X^2-pX)$ is in $\frakQ$. This question is tightly connected to a problem of finding a finite group and its two non-equivalent representations $\rho_1, \rho_2$ over $W(k)$ whose reductions are equivalent and satisfy certain restrictions. More precisely, we want $\End(\brho) = k$. This is to guarantee the existence of the universal deformation ring. Additionally, we want that for some $g$, $\chi_1(g) - \chi_2(g) \in (p) \setminus (p^2)$, where $\chi_1, \chi_2$ are the corresponding characters. The first attempt would be to "artificially" produce a finite group with an irreducible $\bbF_p$-representation to force the conditions to hold. This is not easy, however, as by \cite{Schmid} this will usually fail for $p$-solvable groups.

Our method is to study carefully two representations of $\SL(2,\bbF_{p^2})$ whose characters are equal on the $p$-regular conjugacy classes. The difficulty is that the reduction of any representation having one of these characters is never a simple representation -- this complicates the proof of existence of the universal deformation ring for the residual representation and prevents us from using the result of Carayol and Serre to conclude that the considered representations not only have characters with values in $\bbZ_p$, but can actually be defined over $\bbZ_p$. On the way to prove the existence of $R_{\brho}$ in our case, we generalize the so-called Ribet's lemma, see Proposition \ref{Ribet'slemma}.
Our main result is the following.
\begin{thm*}
Let $k=\bbF_{p^s}$ for even $s$. Then for any $p$, the ring $W(k)[\![X]\!]/(X^2-pX)$ is in $\frakQ$.
\end{thm*}
\section{Deformations of representations}

In this section, we recall basic notions of the theory of deformations of group representations. In the following, $k$ will denote a finite field and $p$ will be its characteristic.

\begin{defn}
Let $k$ be a finite field. Let $\hcalC$ denote the category of all complete noetherian local commutative rings with residue field $k$.
Morphisms of $\hcalC$ are the local ring homomorphisms inducing the identity
on $k$. Denote by $\calC$ the full subcategory of artinian rings in $\hcalC$.
\end{defn}

Let us gather some facts on $\hcalC$ and $\calC$. Let $W(k)$ denote the ring of Witt vectors over $k$.
\begin{prop}
\begin{enumerate}
\item The category $\calC$ coincides with the full subcategory of finite rings in
$\hcalC$.
\item \label{cohen-strocture} Every $R \in \hcalC$ is a quotient of a power series ring in
finitely many variables over $W(k)$.
\item \label{cohen-structure-cor} All rings in $\hcalC$ have a natural $W(k)$-algebra structure and homomorphisms in $\hcalC$ coincide with local $W(k)$-algebra homomorphisms.
\end{enumerate}
\end{prop}
\begin{proof}
Parts (\ref{cohen-strocture}) and (\ref{cohen-structure-cor}) are consequences of the Cohen’s structure theorem; see \cite[\S 29]{matsumura}.
\end{proof}

Let us move to the definition of universal deformation rings.
Given a ring $R\in \hcalC$, we will denote the functor $\Hom_{\hcalC}(R,-) : \hcalC \rightarrow \mathrm{Sets}$ by $h_R$.

To see the full context of the deformation theory, we are going to temporarily consider representations of profinite groups. Let $G$ be a profinite group and let $\brho : G \rightarrow \GL_n(k)$ be a continuous representation (where $k$ and, consequently, $\GL_n(k)$ are considered with the discrete topology).
\begin{defn}
We define a \emph{lift} of $\brho$ to $R \in \hcalC$ as a continuous group homomorphism $\brho : G \rightarrow \GL_n (R)$ such that $\brho = GL_n(\pi_{\frakp_R}) \circ \rho$. By $GL_n(\pi_{\frakp_R})$ we mean the map $\GL_n(R) \rightarrow \GL_n(k)$ induced by the reduction map $\pi_{\frakp_R}: R \rightarrow k$. We will call two lifts $\rho$, $\rho'$ \emph{strictly equivalent} if there exists $K \in \ker \GL_n(\pi_{\frakp_R})$ such that $\rho'= K\rho K^{-1}$. The set of resulting
equivalence classes will be denoted by $\Def_{\brho} (R)$ and its elements will be called \emph{deformations} of $\brho$ to $R$.
A morphism $f : R \rightarrow R'$ induces a map $\GL_n(f) : \GL_n(R) \rightarrow \GL_n(R')$ that preserves strict equivalence classes, and so gives rise to a map $\Def_{\brho}(R) \rightarrow \Def_{\brho}(R')$. Thus, $\Def_{\brho}$ defines a functor $\hcalC \rightarrow \mathrm{Sets}$ called the \emph{deformation functor} for $\brho$. If the functor $\Def_{\brho}$ is representable by $R_{\brho} \in \hcalC$, we call $R_{\brho}$ the \emph{universal deformation ring} for $\brho$.
\end{defn}
\begin{rmk}
The definition of the deformation functor can be rewritten in the language of topological $k[\![G]\!]$ and $R[\![G]\!]$-modules, where $R[\![G]\!]=\varprojlim_{\textrm{open }N \lhd G} R[G/N]$ (and similarly for $k$).
\end{rmk}

We now proceed to the question of representability. We denote by $\mathrm{CHom}(\cdot,\cdot)$ the set of continuous homomorphisms.
\begin{thm}\label{representabilty}
Let $\brho : G \rightarrow \GL(V)$ be a continuous representation over $k$ of a profinite group. Assume that
\begin{enumerate}
\item $\mathrm{CHom}(\ker \brho , \bbZ / p\bbZ)$ is finite;
\item If $M \in \End_k(V)$ satisfies $\brho(g) M\brho(g)^{-1}=M$, $g \in G$, then there exists $\lambda \in k$ such that $M=\lambda \cdot \Id$.
\end{enumerate}
Then the functor $\Def_{\brho}$ is representable.
\end{thm}
\begin{proof}
See \cite[Prop. 7.1]{dSL}.
\end{proof}

Observe that the first condition is automatically satisfied for finite groups $G$, which will be the case of our interest.

\section{The inverse problems}
Fix a finite field $k$ and consider the associated category $\hcalC$. The inverse problem is  the following question:
\begin{center}
\emph{Which rings of $\hcalC$ occur as universal deformation rings for some profinite group and its representation over $k$?}
\end{center}
This was an open question for many years. For example, for some time it was conjectured that only complete intersections can occur as universal deformation rings, but  counterexamples were found. The final answer to the question is due to Dorobisz in his PhD thesis, see \cite{dorobisz} and \cite{dorobisz2} (and independently due to Eardley and Manoharmayum, see \cite{EaM}), namely
\begin{thm}
All the rings in $\hcalC$ can occur as universal deformation rings.

More precisely, let $R$ be a complete noetherian local ring with a finite
residue field $k$, $n \geq 2$ and consider the natural representation $\brho$ of $\SL_n(R)$
in $\GL_n(k)$. Then $R$ is the universal deformation ring of $\brho$ if and only if
$(n,k) \notin \{(2,\bbF_2),(2,\bbF_3),(2,\bbF_5), (3,\bbF_2)\}$.
\end{thm}
\begin{proof}
See \cite[Theorem 1.1]{dorobisz2}.
\end{proof}
Recall that in the statement of the inverse problem we allow all profinite groups and their representations. Dorobisz' proof indeed uses this fact, as the group $\SL_n(R)$ is not finite for an infinite ring $R$ and $n \geq 2$.
This leaves open the following question:
\begin{center}
\emph{Which rings of $\hcalC$ occur as universal deformation rings for some \textbf{finite} group and its representation over $k$?}
\end{center}
Let us also introduce a convenient notation:
\begin{notation}
\begin{enumerate}
\item Denote by $\frakU$ the class of rings in $\hcalC$ that occur as universal deformation rings for a representation of a finite group.
\item Denote by $\frakQ$ the class of rings in $\hcalC$ that occur as a quotient of a ring in $\frakU$.
\end{enumerate}
\end{notation}
It turns out that the answer to this question cannot be the same as before, i.e. there are rings in $\hcalC$ that are not in $\frakU$. The reason is simple -- there are only countably many finite groups and their representations (up to isomorphism) over $k$ while it can be shown that the class $\mathrm{Ob}(\hcalC)$ is uncountable (\cite[Prop. 6.3]{dorobisz}). On the other hand, by the theorem of Dorobisz, we see that all finite rings of $\hcalC$ lie in $\frakU$. But there are more rings in $\frakU$. For example $\bbZ_p \in \frakU$, when $k = \bbF_p$ (this can be obtained by taking any finite group of order prime to $p$ and its irreducible representation over $\bbF_p$, see \cite[Lemma 3.2]{dorobisz2}). Many rings are ruled out from being in $\frakU$ by the following result.
For a ring $R$, denote by $T_{p^\infty}(R)$ the $p$-torsion subgroup of $(R,+)$, i.e. $T_{p^\infty}(R)=\bigcup_{i=1}^\infty \{r \in R |p^ir=0\}$.

\begin{thm} (\cite[Theorem 6.30.]{dorobisz})
Let $R \in \hcalC$ be of characteristic zero and assume $R \in \frakU$. Then
$R/T_{p^\infty}(R)$ is reduced and of Krull dimension $1$.
\end{thm}
Despite this result, it is not clear which among the rings of characteristic $0$ and such that
$R/T_{p^\infty}(R)$ is reduced and of Krull dimension $1$ occur in $\frakU$. In his thesis, Dorobisz asks the following question (\cite[Question 6.39]{dorobisz}).
\begin{quest}
Which of the following rings are in $\frakU$ (are in $\frakQ$)?
\begin{enumerate}
\item $W(k)[\sqrt[r]{p}]$, $r>1$
\item $W(k)[\![X]\!]/(X^2-p^rX)$, $r\geq 1$
\item $W(k)[\![X]\!]/(p^rX)$, $r \geq 1$
\end{enumerate}
\end{quest}

The main result of this article is to prove the following
\begin{center}
\emph{Let $p$ be a prime. Then the ring $W(\bbF_{p^s})[\![X]\!](X^2-pX)$ is in $\frakQ$ for every even $s$.}
\end{center}
Actually, the case $p=2$ is quite easy and we can skip the assumption that $s$ is even in this case. Most of our work is devoted to solving the case when $p$ is odd.

\section{$\SL(2,q)$ and its representation theory}
\subsection{Representation theory of $\SL(2,q)$ in characteristic $0$}
\begin{defn}
In this and the following sections, $q$ will denote a power of an odd prime number $p$, unless stated otherwise. By $\SL(2,q)$ or $\SL_2(\bbF_q)$ we mean the (finite) group of all the $2 \times 2$ matrices over the finite field $\bbF_q$ that have determinant equal $1$. By $\GL(2,q)$ or $\GL_2(\bbF_q)$ we mean the (finite) group of all invertible $2 \times 2$ matrices over $\bbF_q$.
\end{defn}
The group $\SL(2,q)$ will play the main role. We start by recalling some basic proprieties of this group and describing some of its representations. More details can be found in \cite[Ch. 1,3,5]{bonna}.
\begin{defn} We define the following subgroups of $\SL(2,q)$:
\begin{gather}
T=\Bigg\{\left( \begin{array}{cc}
a & \\
  & a^{-1}
\end{array}\right)\Biggm\vert a \in \bbF_q^*\Bigg\}, \text{  }
U= \Bigg\{\left( \begin{array}{cc}
1 & a \\
  & 1  \end{array}\right)\Biggm\vert a \in \bbF_q\Bigg\}, \text{  }
B=\Bigg\{\left( \begin{array}{cc}
a & b \\
  & a^{-1}  \end{array}\right)\Biggm\vert a \in \bbF_q^*, b \in \bbF_q\Bigg\}.
\end{gather}

\end{defn}
For a natural number $m$, we will denote by $\mu_m$ the group of $m$-th roots of unity in $\bbbF_p$.
We have the obvious isomorphisms:
\begin{gather*}
\bd : \mu_{q-1}=F_q^* \simeq T, \text{  }
a \mapsto \left( \begin{array}{cc}
a & \\
  & a^{-1}
\end{array}\right) \text{  and  }  \bu : (\bbF_q,+) \simeq U, \text{   } a \mapsto \left( \begin{array}{cc}
1 & a\\
  & 1
\end{array}\right)
\end{gather*}

Let us define another subgroup of $\SL(2,q)$. Consider $\bbF_{q^2}$ as a two-dimensional vector space over $\bbF_q$ and fix some basis of it to get an isomorphism $$\bd' : \GL_{\bbF_q}(\bbF_{q^2}) \stackrel{\sim}{\rightarrow} \GL_2(\bbF_q).$$ Elements of $\bbF_{q^2}^*=\mu_{q^2-1}$ act $\bbF_q$-linearly on the space $\bbF_{q^2}$. In this way we get $\mu_{q+1}\subset \GL_2(\bbF_q)$ and it can be checked that the image lands in $\SL_2(\bbF_q)$. So we get an isomorphism of $\mu_{q+1}$ with a subgroup of $G$.

Let us gather some basic facts about SL$(2,q)$. For convenience, in this subsection we will be sometimes denoting $G=SL(2,q)$ and $I_2=\left( \begin{array}{cc}
1 &  \\
  & 1   \end{array}\right)$.
\begin{fact}
\begin{enumerate}
\item $|SL(2,q)|=q(q-1)(q+1)$.
\item $B=U \rtimes T$, where $\rtimes$ denotes the semidirect product.
\item $Z(G)=\{I_2,-I_2\}$.
\item Denote $s=\left( \begin{array}{cc}
 & -1 \\
1 &   \end{array}\right)$. Then $G=B \sqcup BsB$ (\emph{Bruhat decomposition}), where $\sqcup$ denotes here and later the disjoint union.
\end{enumerate}
\end{fact}
\begin{proof}
For the Bruhat decomposition see \cite[§ 1.1.1]{bonna}.
\end{proof}
Let us describe the conjugacy classes of $G$. Denote by $\equiv$ the relation on $\bbF^*$ defined by $x \equiv y$ if $y \in \{x,x^{-1}\}$ and fix an element $c_0 \in \bbF_q$ which is not a square
in $\bbF_q$ (which is possible as $q$ is odd). Set
\begin{displaymath}
u_+ = \left( \begin{array}{cc}
1 & 1\\
  & 1
\end{array}\right) \textrm{ and } u_-=\left( \begin{array}{cc}
1 & c_0\\
  & 1
\end{array}\right)
\end{displaymath}
\begin{fact}
The group G consists of $q+4$ conjugacy classes. Each conjugacy class is represented by exactly one element of the following set.
\begin{displaymath}
\{I_2,-I_2\}\sqcup \{u_+,u_-,-u_+,-u_-\} \sqcup \{\bd(a)|a \in (\mu_{q-1}\setminus \{1,-1\})/\equiv \} \sqcup \{\bd'(\ksi)|\ksi \in (\mu_{q+1}\setminus \{1,-1\})/\equiv \}
\end{displaymath}
\end{fact}
\begin{proof}
See \cite[Theorem 1.3.3]{bonna}.
\end{proof}

We now proceed to the description of two important characters of $G$. Fix some generator $\epsilon$ of the character group $\Hom(T,\bbC^*)$ of $T$. Let $\alpha$ be a character of $T$. Restrict $\alpha$ to $B$ via the projection $B \rightarrow T$ to get a character $\alpha_B$ of $B$. Consider its induced character $$R(\alpha)=\Ind_B^G\alpha_B$$ on $G$. Now, let us focus on $\alpha_0=\epsilon^{(q-1)/2}$, which is the only non-trivial character whose square is equal to $1$. It turns out that the induced character $R(\alpha_0)$ is a sum of two irreducible characters which are equal on $p$-regular conjugacy classes. Moreover, these characters attain values in $\bbZ$ when $q$ is a square. In the statement below we use that as $\SL(2,q)$ is a normal subgroup of $\GL(2,q)$, we get an action of $\GL(2,q)$ on conjugacy classes of $\SL(2,q)$ and so also on characters.
\begin{prop}
\begin{enumerate}
    \item $R(\alpha_0)$ is a direct sum of two irreducible characters $R_+(\alpha_0)$ and $R_-(\alpha_0)$.
    \item The characters $R_+(\alpha_0)$ and $R_-(\alpha_0)$ are permuted by the action of $\GL(2,q)$. In particular, they have the same dimension.
    \item The characters of $R_{\pm}(\alpha_0)$ are summed up in the following table:
    \begin{table}[htb]
\centering
\caption{Values of the character $R_{\pm}(\alpha_0)$}\label{valuesofR_pm}
\begin{tabular}{|l|c|c|c|c|}
  \hline 
  & $\epsilon I_2$ & $\textbf{d}(a)$ & $\textbf{d'}(\ksi)$ & $\epsilon u_\tau$\\
  & $\epsilon \in \{\pm1\}$ & $a \in \mu_{q-1} \setminus \{\pm1\}$ & $\ksi \in \mu_{q+1}\setminus \{\pm1\}$ & $\epsilon \in \{\pm 1\}$,  $\tau \in \{\pm 1\}$\\
  \hline
 
 $R_\sigma(\alpha), \sigma \in \{\pm\}$ & $\frac{(q+1)\alpha_0(\epsilon)}{2}$ & $\alpha_0(a)$ & $0$ & $\alpha_0(\epsilon)\frac{1+\sigma \tau \sqrt{\alpha_0(-1)q}}{2}$ \\[1ex]
  \hline
\end{tabular}
\end{table}
\FloatBarrier
    
\end{enumerate}

\end{prop}

\begin{proof}
For the computations, see \cite[\S 3., \S 5.]{bonna}. The full character table of $\SL(2,q)$ can be found in \cite[Table 5.4]{bonna}).
\end{proof}

\begin{rmk}
To make sense of the table, one needs to fix a square root of $\alpha_0(-1)q$ in $\overline{\bbQ}_p$ in a correct way. It will not matter in our applications and can be safely ignored. We only mention that, as explained in \cite[\S 5.2.3]{bonna}, this is done by fixing a non-trivial linear character $\chi_+$ of $\bbF_q^+$ and setting
\begin{displaymath}
\sqrt{\alpha_0(-1)q} = \sum_{z \in \bbF_q^*} \alpha_0(z)\chi_+(z).
\end{displaymath}
\end{rmk}

Observe that $\alpha_0$ takes values in $\{\pm 1 \}$, so the characters $R_{\pm}(\alpha_0)$ take values in $\bbQ_p[\sqrt{\alpha_0(-1)q}]$. In fact, more is true.

\begin{prop}\label{fieldofdef}
The representations corresponding to $R_+(\alpha_0)$ and $R_-(\alpha_0)$ can be defined over $\bbQ_p[\sqrt{\alpha_0(-1)q}]$.
\end{prop}
\begin{proof}
In general, these kind of problems can be approached using the so-called Schur index (see \cite[\S 10]{isaacs}). Let us give a direct proof. It was suggested to me by Ehud Meir as an answer to my question on MathOverflow (see \cite{MO}). Denote $K=\bbQ_p[\sqrt{\alpha_0(-1)q}]$ and $d=(q+1)/2$. Observe that the representation of $T$ attached to $\alpha_0$ can be defined over $\bbQ_p$. From this it follows easily that the representation $R(\alpha_0)$ can be defined over $\bbQ_p$ as well, and so also over $K$. Choose such a representation $V$ of $G$ over $K$ with character $R(\alpha_0)$. We will use \cite[Thm. 3.6.2]{webb}, which gives an explicit description of the primitive central idempotents appearing in the  Artin-Wedderburn theorem. It implies that $e_+=\frac{d}{|G|}\sum_{g \in G} R_+(\alpha_0)(g^{-1})g$ is an idempotent of $KG$. We define $e_-$ in an analogous way. Observe that $e_+V, e_-V \subset V$ are $KG$-submodules of $V$ and by taking algebraic closures, we see that $V=e_+V \oplus e_-V $. Indeed, $V\otimes \overline{K} \simeq V_+ \oplus V_-$, where $V_{\pm}$ correspond to $R_\pm(\alpha_0)$ and from \cite[Thm. 3.6.2]{webb} we know that $V_+$ is the unique simple $\overline{K}Ge_+$-module and so we see that $\overline{K}Ge_+$ annihilates $V_-$ (and symmetrically for $e_-$). Thus we see that $V\otimes \overline{K} = (e_+ V\otimes \overline{K})\oplus (e_- V\otimes \overline{K})$ from which we easily conclude the claim. The modules $e_\pm V$ have characters respectively $R_\pm(\alpha_0)$ and they are the desired $G$-modules defined over $K$.
\end{proof}
\begin{rmk}
We will see soon that reductions of $R_{\pm}(\alpha_0)$ with respect to any lattice are not simple. Otherwise, we could use the following result of Carayol and Serre:
\begin{thm}\label{carayol}
Let $H$ be a finite group and $A$ a ring in $\calC$. Let $\rho: H \rightarrow GL_n(A)$ be a representation whose reduction is absolutely irreducible. Assume that the values of its character lie in some $A_0 \subset A$ which is in $\calC$. Then there exists a representation $\rho_0:H \rightarrow GL_n(A_0)$ such that $\rho = \rho_0 \otimes_{A_0} A$.
\end{thm} 
\begin{proof}
See \S6 of Mazur's article in \cite{FLT} for a discussion and proof.
\end{proof}
\end{rmk}

\subsection{Modular representation theory of SL$(2,q)$}
We now proceed to the study of representations of $G$ in characteristic $p$ (i.e. "equal characteristic case"). Observe that there is a natural two-dimensional representation of $G=\SL(2,q)$ over $\bbF_q$, i.e. the one given by the inclusion $\SL_2(\bbF_q) \rightarrow \GL_2(\bbF_q)$. Let us denote it by $V$. We then also have higher exterior powers of this representation $\Lambda^nV=\Delta(n)$. These representations can be also seen as representations of $G$ on the set of homogenous polynomials $\oplus_{m=0}^n \bbF_qx^{n-m}y^m$ with the natural action of $G$ given by the formula: $g=\left( \begin{array}{cc}
a & b\\
 c & d
\end{array}\right)$ acts by $x^ky^l \mapsto (ax+cy)^k(bx+dy)^l$. Let by $I(n)$ denote the set of $\{m \in \bbN|m_i \leq n_i\}$, where $m_i$ and $n_i$ are coefficients in the $p$-adic expansions of the numbers $m$ and $n$ respectively, i.e. $m=\sum_i m_ip^i$ and similarly for $n$. It can be computed that $L(n)=\oplus_{m \in I(n)} \bbF_q x^{n-m}y^m$ is a submodule of $\Delta(n)$. More precisely, $L(n)$ is a submodule of $\Delta(n)$ generated by $x^n$ (see \cite[§ 10.1.2]{bonna}). It turns out that $L(n)_{0 \leq n  \leq q-1}$ furnish all the irreducible representations over $\bbF_q$.
\begin{thm}\label{lsindelta}
\begin{enumerate}
\item $(L(n) \otimes \bbbF_q)_{0 \leq n \leq q-1}$ is a set of representatives of the isomorphism
classes of simple $\bbbF_q G$-modules. It follows that $(L(n))_{0 \leq n \leq q-1}$  is a set of representatives of the isomorphism
classes of simple $\bbF_q G$-modules.
\item $\langle \Delta(n): L(m)\rangle_G =\mathbf{\Delta}_{m,n}$, where $\mathbf{\Delta}_{m,n}$ is a number determined in the following way:  define recursively
the set $\mathscr{E}(n)$ as follows:
\begin{displaymath}
\mathscr{E}(n) = \left\{ \begin{array}{ll}
\{0\} & \textrm{if $0 \leq n \leq p-1$}\\
p\mathscr{E}\big (\frac{n-n_0}{p} \big ) & \textrm{if $n \geq p$ and $n_0 = p-1$}\\
p\mathscr{E}\big (\frac{n-n_0}{p} \big ) \sqcup n_0 +1 + p\mathscr{E}\big (\frac{n-n_0-p}{p} \big ) & \textrm{if $n \geq p$ and $n_0 \leq p-2$}
\end{array} \right.
\end{displaymath}
Then $\mathbf{\Delta}_{ m,n} \in \{0,1\}$ and $\mathbf{\Delta}_{ m,n} = 1$ if and only if $m \in  n-2\mathscr{E}(n)$. Recall that $n_0$ denotes the first digit of the $p$-adic expansion of $n$. Here, $\langle \Delta(n): L(m)\rangle_G$ denotes the multiplicity of $L(m)$ as a factor in a Jordan-H\"older series of $\Delta(n)$.
\end{enumerate}

\end{thm}
\begin{proof}
See \cite[Theorem 10.1.8.]{bonna}.
\end{proof}

For a field $L$, let $R_L(G)$ denote the Grothendieck group of finitely generated $L[G]$-modules. For a field $K$ complete with respect to a discrete valuation and its residue field $k$, there is a ring homomorphism $d : R_K(G) \rightarrow R_k(G)$, coming from a choice of a stable lattice (but independent from the choice), see \cite[\S 15.]{serre}.
\begin{prop}
Let $K=\bbQ_p(\sqrt{\alpha_0(-1)q})$. We have
\begin{displaymath}
d(R_{\pm}(\alpha_0))=[\Delta((q-1)/2)]
\end{displaymath}
in $R_k(G)$.
\end{prop}
\begin{proof}
See \cite[Proposition 10.2.9]{bonna}.
\end{proof}

\begin{cor}\label{reductionsofR_pm}
Let $q=p^2$, so $G=\SL(2,p^2)$. Then 
\begin{displaymath}
d(R_{\pm}(\alpha_0))=[L((p^2-1)/2)]+[L((p+1)(p-3)/2)] \text{ in $R_k(G)$.}
\end{displaymath}
\end{cor}
\begin{proof}
This is a direct corollary of the last two results.
\end{proof}

\section{The main result}
We consider the ring $R=W(\bbF_{p^s})[\![X]\!](X^2-pX)$. The first observation is that $R$ is isomorphic to the fiber product $W(\bbF_{p^s}) \times_{\bbF_{p^s}} W(\bbF_{p^s})$, i.e. to the subring of $W(\bbF_{p^s})[\![X]\!](X^2-pX)\times W(\bbF_{p^s})[\![X]\!](X^2-pX)$ of pairs $(a,b)$ with the same reduction modulo $(p)$.
\begin{lemma}\label{isowithfiberprod}
We have an isomorphism
\begin{displaymath}
W(\bbF_{p^s})[\![X]\!]/(X^2-pX) \simeq W(\bbF_{p^s}) \times_{\bbF_{p^s}} W(\bbF_{p^s})
\end{displaymath}
\begin{proof}
Observe first that $W(\bbF_{p^s})[\![X]\!]/(X^2-pX) = W(\bbF_{p^s})[X]/(X^2-pX)$. We define a map from $R$ to $W(\bbF_{p^s}) \times_{\bbF_p} W(\bbF_{p^s})$ by $X \mapsto (0,p)$. We see readily that it is well defined (i.e. $X^2-pX$ maps to $0$) and that it is surjective. It is also injective, as in the kernel consists of exactly those polynomials which are divisible by $X$ and by $X-p$. We get the result as $W(\bbF_{p^s})[X]$ is factorial.
\end{proof}

\end{lemma}
Our strategy is to find, for any $p$, a finite group $G$ and two representations of $G$ $\rho_1$ and $\rho_2$ over $W(\bbF_{p^s})$ whose reductions are the same (we denote them by $\brho$), satisfy the condition $\End_{\bk}(\bar{\rho})=\bk$ and whose traces evaluated at some conjugacy class is not divisible by $p^2$. This will imply that there exists a universal deformation ring of $\brho$ and that it admits two different maps to $W(\bbF_{p^s})$ which are equal modulo $p$. The condition involving traces will guarantee surjectivity of the map $R_{\brho} \rightarrow W(\bbF_{p^s}) \times_{\bbF_p} W(\bbF_{p^s})$.

We will need the following generalization of the so-called Ribet's lemma (see \cite[Prop. 2.1]{ribet}) from the $2$-dimensional case to higher dimensions. The proof is virtually the same but using block matrices. Recall that by $d$ we denoted the map $R_K(G) \rightarrow R_k(G)$.
\begin{prop}\label{Ribet'slemma}
Let $G$ be a finite group and let $\rho:  G \rightarrow \GL(V)$ be an irreducible $n$-dimensional representation over a $p$-adic field $K$. Assume $d([\rho])=[\phi_1]+[\phi_2] \in R_k(G)$ for two (not necessarily different) irreducible representations $\phi_1$ and $\phi_2$ of $G$ over $k$. Then there exists a $G$-stable $\calO_K$-lattice for which the reduction is of the form $\left( \begin{array}{cc}
\phi_1 & *\\
  & \phi_2
\end{array}\right)$ (this is a block notation, $\phi_i$ are now matrices that we obtain by fixing a lattice and $*$ is an unknown submatrix of a suitable dimension) and is \textbf{not} isomorphic to the direct sum of $\phi_1$ and $\phi_2$.
\end{prop}
\begin{proof}
We fix any $G$-stable $\calO_K$-lattice $V_0$ in $V$. We can assume that the reduction with respect to this lattice is of the form $\left( \begin{array}{cc}
\phi_1 & *\\
  & \phi_2
\end{array}\right)$ or $\left( \begin{array}{cc}
\phi_1 & \\
 * & \phi_2
\end{array}\right)$.
This is because the reduction has to contain as subrepresentation $\phi_1$ or $\phi_2$ and we can then choose a suitable basis of $V\otimes_{\calO_K} k$ to get one of those forms and lift it in any way to an $\calO_K$-basis ov $V_0$. We claim that we can choose an invariant lattice in such a way that we are in the first case (and will fix such a representation $\rho_0 : G \rightarrow \GL_n(\calO_K)$).  Indeed, an easy calculation shows that conjugating some matrix by $P=\left( \begin{array}{cc}
I & \\
  & \pi I
\end{array}\right)$ (where $\pi$ is a uniformizer of $\calO_K$) results in the same matrix but with the right top block divided by $\pi$ and the bottom-left block multiplied by $\pi$, i.e. $\left( \begin{array}{cc}
1 & \\
  & \pi I
\end{array}\right)\left( \begin{array}{cc}
A & B\\
 C & D
\end{array}\right)\left( \begin{array}{cc}
I & \\
  & \pi^{-1} I
\end{array}\right)=\left( \begin{array}{cc}
A & \pi^{-1}B\\
 \pi C & D
\end{array}\right)$. So in the second case we can conjugate $\rho$ by $P$ to get a lattice with a reduction of the first form as needed. Thus, we saw that we can assume that the bottom-left block of $\rho$ is divisible by $\pi$. We assume moreover that the reduction with respect to any lattice that is of the form $\left( \begin{array}{cc}
\phi_1 & * \\
  & \phi_2
\end{array}\right)$ is isomorphic to a direct product of $\phi_1$ and $\phi_2$ (we will say that it is \emph{split} further on in this proof) and aim at finding a contradiction. We now want to show inductively that there exists a converging sequence of matrices $B_i\in \GL_n(K)$ of the form $\left( \begin{array}{cc}
I & T_i\\
  & I
\end{array}\right)$ (the sizes of the blocks are the same as previously) such that conjugating the representation $\rho_0$ by $B_i\in \GL_n(K)$ we will get a representations consisting of matrices such that the bottom-left block is divisible by $\pi$ and the upper-right one is divisible by $\pi^i$. This will finish the proof as the matrices will converge to some matrix $B$ and conjugating the representation $\rho_0$ by this matrix will give a representation with the upper-right block equal to zero, which will contradict the irreducibility of $\rho$. We know that the claim is true for $i=0$. Assume it now for some $i$. We want to show it for $i+1$.

We have assumed that the reduction of $\rho$ attached to any lattice that is of the form $\left( \begin{array}{cc}
\phi_1 & *\\
  & \phi_2
\end{array}\right)$ is also split. This applies in particular to the representation $P^iB_i\rho_0 B_i^{-1}P^{-i}$, so we can choose a matrix of the form $Q_i=\left( \begin{array}{cc}
I & U_i\\
  & I
\end{array}\right) \in \GL_n(\calO_K)$ so that conjugating the  $P^iB_i\rho_0 B_i^{-1}P^{-i}$ by this matrix will have the reduction equal to $\left( \begin{array}{cc}
\phi_1 & \\
  & \phi_2
\end{array}\right)$. To see the last claim: if the reduction is  of the form $\left( \begin{array}{cc}
\phi_1 & *\\
  & \phi_2
\end{array}\right)$ and is split, then we know that the associated extension $0 \rightarrow \phi_1 \rightarrow \overline{P^iB_i\rho_0 B_i^{-1}P^{-i}} \rightarrow \phi_2 \rightarrow 0$ is split due to the fact that there exists some isomorphism $\overline{P^iB_i\rho_0 B_i^{-1}P^{-i}} \simeq \phi_1 \oplus \phi_2$ (as we assumed that $\brho$ is split) and the assumption of irreducibility of $\phi_i$'s. Indeed, we  use that a map between irreducible modules is either an isomorphism or the zero map: due to the splitting, both $\phi_1$ and $\phi_2$ are submodules of $\overline{P^iB_i\rho_0 B_i^{-1}P^{-i}}$ and one of them must map isomorphically on $\phi_2$ via the map $\overline{P^iB_i\rho_0 B_i^{-1}P^{-i}} \rightarrow \phi_2$ in the exact sequence. This gives the desired section which we denote by $s:\phi_2 \rightarrow \overline{P^iB_i\rho_0 B_i^{-1}P^{-i}}$. If the underlying vector space of the representation $\overline{P^iB_i\rho_0 B_i^{-1}P^{-i}}$ can be decomposed as $\bar{V}\simeq V_1 \oplus V_2$ (where $V_1$ corresponds to $\phi_1$), then the base change map $\bar{V}\simeq V_1 \oplus V_2 \rightarrow \bar{V}$, $(v_1,v_2) \mapsto v_1 + s\circ p(v_2)$ is given by conjugating by a matrix of the form $\bar{Q}_i=\left( \begin{array}{cc}
I & \bar{U}_i\\
  & I
\end{array}\right) \in \GL_n(k)$ as it is the identity on $V_1$ and $s\circ p (v_2) - v_2$ is in $V_1$ for $v_2 \in V_2$. This base change gives an isomorphism of $\bar{V}$ with a direct sum of $k[G]$-modules $V_1 \oplus s\circ p (\bar{V})$, so we have proved the last claim (by lifting $\bar{Q}_i$ to a matrix $Q_i\in \GL_n(\calO_K)$ of the form $\left( \begin{array}{cc}
I & U_i\\
  & I
\end{array}\right)$). This means that $Q_iP^iB_i\rho_0 B_i^{-1}P^{-i}Q_i^{-1}$ has the upper-right block divisible by $\pi$ and bottom-left block divisible by $\pi^{i+1}$, which shows that the matrix $B_{i+1}=P^{-i}Q_iP^iB_i$ is a candidate for the next matrix in the series. We compute
\begin{displaymath}
B_{i+1}=P^{-i}Q_iP^iB_i= \left( \begin{array}{cc}
I & \\
  & \pi^{-i}I
\end{array}\right)
\left( \begin{array}{cc}
I & U_i\\
  & I
\end{array}\right)
\left( \begin{array}{cc}
I & \\
  & \pi^i I
\end{array}\right)
\left( \begin{array}{cc}
I & T_i\\
  & I
\end{array}\right) =
\left( \begin{array}{cc}
I &T_i + \pi^i U_i\\
  & I
\end{array}\right)
\end{displaymath}
which shows that $B_{i+1}$ is of the desired form. Furthermore, it is clear that the matrices $B_i$ converge.
\end{proof}
We will use the following calculation of the $\Ext$ groups.
\begin{fact}\label{calcofExt}
We have
\begin{displaymath}
\Ext_{\bbbF_pG}(L((p^2-1)/2)\otimes \bbbF_p, L((p+1)(p-3)/2)\otimes \bbbF_p)=\bbbF_p
\end{displaymath}
\end{fact}
\begin{proof}
This follows immediately from \cite[Corollary 4.5]{modulesofSL2}
\end{proof}

Recall the following fact on the $\Ext$ groups over noncommutative rings. For a ring $R$, let $Z(R)$ denote its center. 
\begin{lemma}
Suppose that $R$ is a ring and $S \subset  R$ its subring contained in $Z(R)$. Then the groups $\Hom_R(A,B)$ and the $\Ext^i_R(A,B)$ are actually $S$-modules. If $\mu : A \rarr A$ and $\nu : B \rarr B$ are multiplication by $s \in S$, so are the induced endomorphisms $\mu^*$ and $\nu^*$ of $\Ext^i_R(A,B)$ for all $i$.
\end{lemma}
\begin{proof}
This is a slight generalization of \cite[Lemma 3.3.6]{weibel}, which is stated for commutative rings $R$. The proof is the same up to obvious modifications.
\end{proof}

\begin{lemma}\label{extandbasechange}
Let $R$ be a (not necessarily commutative) noetherian ring and $k \subset Z(R)$ be a field. Let $M$, $N$ be two finitely generated $R$-modules and $k' \supset k$ a field extension. Then
\begin{displaymath}
\Ext_{R \otimes_k k'}(M\otimes_k k',N\otimes_k k')\simeq \Ext_R(M,N)\otimes_kk'
\end{displaymath}
\end{lemma}
\begin{proof}
\begin{comment}
Choose a resolution $F_\bullet \rightarrow M$ by finitely generated free $R$-modules. Then $F_\bullet \otimes_k k' \rightarrow M \otimes_k k'$ is a resolution of $M \otimes_k k'$ by finitely generated free $R \otimes_k k'$-modules. Because $-\otimes_k k'$ is an exact functor from $R$-modules to $R \otimes_k k'$-modules,
\begin{align*}
\Ext^i_R(M,N)\otimes_k k' = H^i(\Hom_R(F_\bullet,N))\otimes_k k' \simeq H^i(\Hom_R(F_\bullet,N)\otimes_k k')\\
\simeq H^i(\Hom_{R \otimes_k k'}(F_\bullet \otimes_k k',N\otimes_k k'))= \Ext^i_{R \otimes_k k'}(M\otimes_k k',N\otimes_k k')
\end{align*}
(compare with \cite[Lemma 3.3.8]{weibel} and \cite[Prop. 3.3.10]{weibel}).
\end{comment}
The proof is virtually the same as the proofs of \cite[Lemma 3.3.8]{weibel} and \cite[Prop. 3.3.10]{weibel}.
\end{proof}

\begin{lemma}\label{extsareiso}
Let $R$ be a (not necessarily commutative) ring and $k \subset Z(R)$ be a field. Let $M_1$, $M_2$ be two irreducible $R$-modules. Suppose that $\Ext_R(M_2,M_1)=k$ and let $W$ and $W'$ be modules representing some non-trivial classes in $\Ext_R(M_2,M_1)$. Then $W\simeq W'$ as $R$-modules.
\end{lemma}
\begin{proof}
Write $W$ as an extension $0 \rightarrow M_1 \stackrel{i}{\rightarrow} W \stackrel{p}{\rightarrow} M_2 \rightarrow 0$. Because of the assumption $\Ext(M_2,M_1)=k$, all the classes of different non-trivial extensions are represented by $0 \rightarrow M_1 \stackrel{c \cdot i}{\rightarrow} W \stackrel{p}{\rightarrow} M_2 \rightarrow 0$, where $c\in k^*$. Indeed, we know that the map $M_1 \stackrel{\cdot c}{\rightarrow} M_1$ induces an analogous map on $\Ext(M_2,M_1)$. On the other hand, via identification of the "Yoneda Ext" with the "abstract Ext" (see  \cite[Theorem 3.4.3]{weibel}), one checks that the map induced by this multiplication corresponds to the following map of extensions
$$(\ksi : 0 \rightarrow M_1 \stackrel{i}{\rightarrow} W \stackrel{p}{\rightarrow} M_2 \rightarrow 0) \mapsto (c \cdot \ksi: 0 \rightarrow M_1 \stackrel{c^{-1} \cdot i}{\rightarrow} W \stackrel{p}{\rightarrow} M_2 \rightarrow 0) .$$
Although different as extensions, the corresponding $R$-modules are all isomorphic (to $W$). So is $W'$, which finishes the proof.
\end{proof}

The last results will turn out to be useful for proving the existence of the universal deformation rings for some representations when we couple them with the following result.
\begin{lemma}\label{end=k}
Assume that $(\brho,\bV)$ is a representation of some group $H$ over a field $k$ such that $\bV^{ss}$ is a direct sum of two different absolutely irreducible modules $V_1$ and $V_2$, but $\bV$ is not semisimple itself. Then $\End_{kH}(\brho)=k$.
\end{lemma}
\begin{proof}
Observe first, that $\bV$ contains one of the $V_i$ as a subrepresentation. By a change of notation we can assume that it is $V_1$ (and we will be denoting this copy by $V_1$ when there is no risk of confusion). In this case $V_2$ is isomorphic to the quotient of $V$ by $V_1$. Let $\phi$ be in $\End_{kH}(\bV) \setminus \{0\}$. We first show that it is an isomorphism. Assume the contrary. Let us look at the kernel of $\phi$. It is either isomorphic to $0$, $V_1$ or $V_2$. If it is isomorphic to $V_2$, then it must intersect $V_1$ trivially (as these are two different irreducible modules). So it maps isomorphically on the quotient by $V_1$, so provides a section, which contradicts the non-semisimplicity of $\bV$. If the kernel is isomorphic to $V_1$, then we look at the image $\mathrm{im} \phi \simeq V_2$ and do the same reasoning. So we know that $\phi$ is an isomorphism. Similarly, we check that $\phi$ maps $V_1$ into $V_1$ (using the fact that $\phi(V_1)\cap V_1$ is either trivial or equal to $V_1$ and that there are no non-trivial maps from $V_1$ to $V_2$). As $V_1$ is absolutely irreducible, the restriction $\phi|_{V_1}$ is equal to $\lambda \Id$ for some scalar $\lambda$. We can now look at $\phi - \lambda \Id$. It is an element of $\End_{kH}(\bV)$, so as we have just seen it is either zero or an isomorphism. As it has a non-trivial kernel, it must be equal to $0$. This finishes the proof.
\end{proof}

Finally, we are ready to prove the main result. Recall that the category $\hcalC$ and the classes $\frakU$ and $\frakQ$ depend on the field $k$.

\begin{thm}\label{main}
Let $k=\bbF_{p^s}$ for even $s$. Then for any (odd) $p$, the ring $W(k)[\![X]\!]/(X^2-pX)$ is in $\frakQ$.
\end{thm}
\begin{proof}
Let $G=\SL(2,p^2)$, i.e. $q=p^2$. Assume $p \neq 2$ (see Lemma \ref{case-p=2} below for the proof in this case). Denote by $K$ the fraction field of $W(k)$. Let $\rho_+', \rho_-' : G \rightarrow \GL_{(q+1)/2}(K)$ be two representations of $G$ corresponding to the characters $R_+(\alpha_0)$ and $R_-(\alpha_0)$ and defined over $K$. It is possible to find such representations by Proposition \ref{fieldofdef}, as in this case these two representations are defined over $\bbQ_p$. In fact, $\alpha_0(-1)=1$ since $4|q-1$.

Denote by $d$ the reduction map $R_K(G) \rightarrow R_k(G)$. We know that $d(\rho_{\pm})=[L((p^2-1)/2)]+[L((p+1)(p-3)/2)]$ and that $L((p^2-1)/2)$, $L((p+1)(p-3)/2)$ are absolutely irreducible and defined over $\bbF_q$ (Theorem \ref{lsindelta}). Thus, the generalization of Ribet's lemma applies here (Proposition \ref{Ribet'slemma}) and we get two representations $\rho_+', \rho_-' : G \rightarrow \GL_{(q+1)/2}(W(k))$ whose reductions give non-trivial elements of $\Ext_{kG}(L((p^2-1)/2),L((p+1)(p-3)/2))$. By Fact \ref{calcofExt} and Lemma \ref{extandbasechange}, we have $\Ext_{kG}(L((p^2-1)/2),L((p+1)(p-3)/2))=k$. By Lemma \ref{extsareiso}, we see that the reductions $\brho_+$ and $\brho_-$ are isomorphic. So, by choosing a different $W(k)$-base for $\rho_+$, we might assume that the reductions are equal (the characters of $\rho_+$ and $\rho_-$ remained unchanged). Denote this common reduction by $\brho$. By Lemma \ref{end=k} and Theorem \ref{representabilty}, we see that the universal deformation ring $R_{\brho}$ for $\brho$ exists. As $\brho$ lifts in two ways ($\rho_+$ and $\rho_-$) to $W(k)$, from the universal property of the deformation rings we get two morphisms $r_+,r_- : R_{\brho} \rightarrow W(k)$ with the same reduction modulo the maximal ideal. So we get a morphism $r : R_{\brho} \rightarrow W(k)\times_k W(k) \simeq W(k)[\![X]\!]/(X^2-pX)$. We claim that $r$ is surjective. As all rings in $\hcalC$ are $W(k)$-algebras and morphisms in $\hcalC$ are $W(k)$-algebra morphisms, we see that the image of $r$ in $W(k)\times_k W(k)$ contains a $W(k)$-submodule $\{(w,w)|w \in W(k)\}$. So, we see that it is enough to check that the image contains an element of the form $(w,w-p)$ for some $w \in W(k)$. Let $\rho: G \rightarrow \GL_{(q+1)/2}(R_{\brho})$ be the universal lift of $\brho$ to $R_{\brho}$. Then $\rho_+=\tilde{r}_+\circ \rho$ and $\rho_-=\tilde{r}_-\circ \rho$, where $\tilde{r}_\pm: \GL_{(q+1)/2}(R_{\brho}) \rightarrow \GL_{(q+1)/2}(W(k))$ denotes the morphism that applies $r_\pm$ to every coefficient of a matrix in $\GL_{(q+1)/2}(R_{\brho})$. So $\Tr(\rho_\pm)=r_\pm(\Tr(\rho))$ as maps on (conjugacy classes of) $G$. But from Table \ref{valuesofR_pm}, we see that $\Tr(\rho_+(u_+))-\Tr(\rho_-(u_+))=\alpha_0(1)\frac{1+\sqrt{\alpha_0(-1)p^2}-(1-\sqrt{\alpha_0(-1)p^2})}{2}=\frac{1+p-(1-p)}{2}=p$. Denote $t=\Tr(\rho(u_+)) \in R_{\brho}$. Thus, $r_+(t)-r_-(t)=p$, which shows that $(r,r-p)$ lies in the image of $R_{\brho} \rightarrow W(k)\times_k W(k)$ for $r=r_+(t)$ and finishes the proof.
\end{proof}
 
Let us deal with the case $p=2$. In this case, we get more and with less effort.
 
\begin{lemma}\label{case-p=2}
Let $s>0$ be an integer. The ring {$W(\bbF_{2^s})[\![X]\!]/(X^2-2X)$} is actually in $\frakU$. 
\end{lemma}
\begin{proof}
This is in fact a part of a more general formula for universal deformations rings of one dimensional representations that is mentioned for example in \cite[Ch. 2.2]{rainone}. Let $\bbZ/2\bbZ$ act trivially on $\bbF_{2^s}$. In fact, this is the only possible action, as $2 \nmid |\bbF_{2^s}^*|$. Then the universal deformation ring of this representation is equal to  $W(\bbF_{2^s})[\![X]\!]/(X^2-2X)$.
 To see it, observe that $W(\bbF_{2^s})[\![X]\!]/(X^2-2X)$ is isomorphic to $W(\bbF_{2^s})[\bbZ/2\bbZ]$ by mapping $X-1 \mapsto \sigma$, where $\sigma$ is the generator of $\bbZ/2\bbZ$. Now, $ W(\bbF_{2^s})[\bbZ/2\bbZ]$  is the universal deformation ring we are looking for, because for any ring $A  \in \hcalC$, to give a representation $\bbZ/2\bbZ \rightarrow \GL_1(A)=A^*$ is the same as to choose an element of $a \in A^*$ that satisfies $a^2=1$. For any such $A$, the set $\{a \in A^*| a^2=1\}$ is clearly parametrized by maps from $W(\bbF_{2^s})[\bbZ/2\bbZ]$ to $A$.
\end{proof}

\begin{rmk} Looking for examples of groups with a pair of different representations that would allow for a similar proof is not easy. Our proof would be simplified if the reductions of $R_{\pm}(\alpha_0)$ were simple. It is hard to construct examples of such groups and representations defined of $\bbQ_p$, as such a group cannot be $p$-solvable. This can be concluded from the main theorem of \cite{Schmid}.

A simpler result in this spirit is the following: this phenomenon is not possible if $p \nmid |G|$ (where $p$ is the characteristic of the residue field we work with). See \cite[Proposition 43]{serre}. An elementary proof can be found in \cite[Lemma 6.10]{dorobisz}.

\end{rmk}

\section*{Acknowledgments}
This work was a part of my Master thesis. I am grateful to my supervisor Dr Jakub Byszewski for introducing me into the topic and many helpful discussions.

%\bibliographystyle{alpha}
%\bibliography{document}

\end{document}